\documentclass[12pt]{amsart}
\usepackage[colorlinks=true,citecolor=blue,linkcolor=blue]{hyperref}
\usepackage{latexsym,amsmath,amssymb}
\usepackage{accents}
  
\title[$L^p$-gradient harmonic maps]{\protect{$L^p$-gradient harmonic maps into spheres and SO(N)}}

\author{Armin Schikorra}

\address{Armin Schikorra, Max-Planck Institut MiS Leipzig, Inselstr. 22, 04103 Leipzig, Germany, {\tt armin.schikorra@mis.mpg.de}}

\thanks{The research leading to these results has received funding from the European Research Council under the European Union's Seventh Framework Programme (FP7/2007-2013) / ERC grant agreement n° 267087.}

plus 6pt minus 12pt
\def\eps{\varepsilon}

\def\N{{\mathbb N}}

\def\S{{\mathbb S}}

\newtheorem{theorem}{Theorem}
\newtheorem{lemma}[theorem]{Lemma}

\newtheorem{proposition}[theorem]{Proposition}

\theoremstyle{definition}

\def\supp{{\rm supp\,}}

\newcommand{\R}{\mathbb{R}}

\newcommand{\Z}{\mathbb{Z}}
\newcommand{\vrac}[1]{\| #1 \|}

\newcommand{\brac}[1]{\left (#1 \right )}

\newcommand{\barint}{
\rule[.036in]{.12in}{.009in}\kern-.16in \displaystyle\int }

\newcommand{\barcal}{\mbox{$ \rule[.036in]{.11in}{.007in}\kern-.128in\int $}}

\def\mvint_#1{\mathchoice
          {\mathop{\vrule width 6pt height 3 pt depth -2.5pt
                  \kern -8pt \intop}\nolimits_{\kern -3pt #1}}%
          {\mathop{\vrule width 5pt height 3 pt depth -2.6pt
                  \kern -6pt \intop}\nolimits_{#1}}%
          {\mathop{\vrule width 5pt height 3 pt depth -2.6pt
                  \kern -6pt \intop}\nolimits_{#1}}%
          {\mathop{\vrule width 5pt height 3 pt depth -2.6pt
                  \kern -6pt \intop}\nolimits_{#1}}}

\numberwithin{theorem}{section} \numberwithin{equation}{section}

\newcommand{\lap}{\Delta }
\newcommand{\aleq}{\precsim}
\newcommand{\aeq}{\approx}
\newcommand{\Rz}{\mathcal{R}}
\newcommand{\laps}[1]{\lap^{\frac{#1}{2}}}
\newcommand{\lapms}[1]{\lap^{-\frac{#1}{2}}}

\newcommand{\cut}{\accentset{\bullet}{\chi}}
\newcommand{\cutA}{\accentset{\circ}{\chi}}

\begin{document}

\sloppy

\subjclass[2010]{58E20, 35B65, 35J60, 35S05}

\sloppy


\begin{abstract}
We consider critical points of the energy
\[
 E(v) := \int_{\R^n} |\nabla^s v|^{\frac{n}{s}},
\]
where $v$ maps locally into the sphere or $SO(N)$, and $\nabla^s = (\partial_1^s,\ldots,\partial_n^s)$ is the formal fractional gradient, i.e. $\partial_\alpha^s$ is a composition of the fractional laplacian with the $\alpha$-th Riesz transform. We show that critical points of this energy are H\"older continuous. 

As a special case, for $s = 1$, we obtain a new, more stable proof of Fuchs and Strzelecki's regularity result of $n$-harmonic maps into the sphere \cite{Strz94,FuchsnharmSphere}, which is interesting on its own.
\end{abstract}

\maketitle

\section{Introduction}
Fix $s \in (0,n)$ and a domain $\Omega \subset \R^n$. In \cite{DSpalphSphere} Da Lio and the author proved H\"older continuity of critical points of the energy
\[
 \tilde{E}_s(v) := \int_{\R^n} |\laps{s} v|^{\frac{n}{s}}
\]
for mappings $v: \R^n \to \R^N$, such that $v(x)$ belongs to the $(N-1)$-dimensional sphere $\S^{N-1}$ for a.e. $x \in \Omega$. Here, $\laps{s}$ denotes the fractional laplacian which for $s \in (0,1)$ is defined as
\[
 \laps{s} v(x) = c_s\ \int_{\R^n} \frac{v(x)-v(y)}{|x-y|^{n+s}}\ dy,
\]
and more generally for $s \in (0,n)$ is defined via its Fourier transform 
\[
 \mathcal{F}(\laps{s} v)(\xi) = c_s |\xi|^s\ \mathcal{F}v(\xi).
\]
A priori, mappings with finite energy $\tilde{E}_s$ belong to BMO, and the structure of the Euler-Lagrange equation is
\[
 \laps{s} (|\laps{s} u|^{\frac{n}{s}-2}\laps{s} u) \in L^1,
\]
a structure which in general allows discontinuous solutions such as $\log \log |x|$ -- the equation is critical.

The motivation for defining an energy like $\tilde{E}$ in \cite{DSpalphSphere} comes from the $n$-harmonic mappings which are critical points of the energy
\[
 E_1(v) := \int_{\R^n} |\nabla v|^{n}, \quad v: \Omega \to \S^{N-1},
\]
whose regularity properties had been studied in the sphere-case by Strzelecki and Fuchs \cite{Strz94,FuchsnharmSphere}. The case where the target sphere $\S^{N-1}$ is replaced by a general closed manifold is largely open, and only under additional assumptions on the solution (which seem unnatural from the point of view of the Calculus of Variations) there are regularity results, cf. \cite{DM10,Kol10,SnHsystemOrlicz}. 
On the other hand, in \cite{Seps} the author showed that the methods from the theory of fractional harmonic maps (i.e. the $L^2$-case) can treat very general Euler-Lagrange equations, which contain as special case both, fractional, i.e. the results of \cite{DR1dMan}, and classical harmonic maps, \cite{Riv06}. Consequently, there is hope to obtain new approaches to the classical energy $E_1$ by investigating the regularizing mechanisms of the fractional harmonic maps.

Nevertheless, the energy $\tilde{E}_1 = \vrac{\laps{1} v}_{L^n}^n$ is different from $E_1 = \vrac{\nabla v}_{L^n}^n$, and it is easier to handle:
Indeed it turned out that the regularity of critical points of $\tilde{E}_s$ in \cite{DSpalphSphere} follows essentially from the theory of fractional harmonic maps into spheres \cite{DR1dSphere, SNHarmS10}, since it is possible to treat $|\laps{s} v|^{\frac{n}{s}-2}$ simply as a weight. In particular, the arguments \cite{DSpalphSphere} fail to recover Strzelecki's/Fuchs' result \cite{Strz94,FuchsnharmSphere} for $E_1$.

Hence, here we are interested in the regularity of critical points of the energy
\[
 E_s(v) := \int_{\R^n} |\nabla^s v|^{\frac{n}{s}}, \quad v: \Omega \to \S^{N-1}.
\]
Here,
\[
 \nabla^s v = \left (\Rz_1 [\laps{s} v], \Rz_2 [\laps{s} v], \ldots, \Rz_n [\laps{s} v]\right )^T,
\]
where $\Rz_\alpha$ is the $\alpha$-th Riesz transform, i.e. the operator with Fourier symbol $i \xi^\alpha/|\xi|$. Let us also remark, that there has recently been some interest in the classical theory of linear and non-linear equations involving $\nabla^s$ \cite{DanielsDs}.

Now $E_s$ contains for $s=1$ the classical $n$-harmonic maps case
\[
 E_1(v) := \int_{\R^n} |\nabla v|^{n}, \quad v: \Omega \to \S^{N-1}.
\]
We then obtain the following theorem
\begin{theorem}\label{th:mainS}
Let $u: \R^n \to \R^N$, $s \in (0,n)$ such that $E_s(u) < \infty$, $u(\Omega) \subset \S^{N-1}$, and assume that $u$ is a critical point of $E_s$, i.e.
\[
 \frac{d}{dt}\Big |_{t = 0} E_s\brac{\frac{u+t\varphi}{|u+t\varphi|} }= 0 \quad \mbox{for any $\varphi \in C_0^\infty(\Omega,\R^N)$.}
\]
Then there exists $\alpha > 0$ such that $u \in C^{0,\alpha}(\Omega)$.
\end{theorem}


As mentioned above, one of our main motivations for this work was to obtain an argument that extends to the classical case of $E_1$. We think that the new proof for \cite{Strz94} following from the proof of Theorem~\ref{th:mainS} is interesting in its own right, since it seems to be more robust than the original proof, or H{\'{e}}lein's proof for the $n=2$ case \cite{Hel90}. In Section~\ref{s:classical} we describe a possibly new angle for a proof of H{\'{e}}lein's \cite{Hel90}, and then describe how our argument for $E_s$ can extend this idea to the $n$-case treated in \cite{Strz94,FuchsnharmSphere}. In particular, in this part we explain the main steps of the proof of Theorem~\ref{th:mainS}.

In the classical case $s = 1$, the arguments for the sphere case \cite{Hel90} can be naturally extended to more general manifolds with symmetries \cite{Hel91}, using Noether's theorem. For the $p$-harmonic case, cf. \cite{ToroWang}. In the case of small $s < 1$, we loose the ability to work with tangent spaces, since $\partial_i u$ is only a distribution. Nevertheless, it not too difficult to extend our argument to a very special case of a Lie Group. Indeed, the case where the unit sphere $\S^{N-1}$ is replaced by the special orthogonal group $SO(N) \subset \R^{N \times N}$ follows along the same lines as Theorem~\ref{th:mainS}.
\begin{theorem}\label{th:mainSO}
Let $u: \R^n \to \R^{N\times N}$, $s \in (0,n)$ such that $E_s(u) < \infty$, $u(\Omega) \subset SO(N)$. Let $\pi: B_\delta(SO(N)) \to SO(N)$ be the orthogonal projection from a tubular neighbourhood onto $SO(N)$, and assume that $u$ is a critical point of $E_s$, i.e.
\[
 \frac{d}{dt}\Big |_{t = 0} E_s\brac{\pi(u+t\varphi)}= 0 \quad \mbox{for any $\varphi \in C_0^\infty(\Omega,\R^{N\times N})$.}
\]
Then there exists $\alpha > 0$ such that $u \in C^{0,\alpha}(\Omega)$.
\end{theorem}
The proof of Theorem~\ref{th:mainS} and Theorem~\ref{th:mainSO} are given in Section~\ref{sec:notaton}.

We introduce some notation in Section~\ref{sec:notaton}, but let us remark here, that we will make frequent use of Einstein's notation, i.e. summing over repeated indices. Also we will denote with greek symbols coordinates in the domain, in particular $\partial_\alpha$ is the $\alpha$'s derivative. Coordinates in the target will be denoted by roman letters, such as $(u^i(x))_{i=1}^N \in \R^N$. With $\aleq$, $\aeq$ we denote estimates up to multiplicative constants (which depend on dimension, $s$, etc., but not on relevant information), i.e. $A \aleq B$ means that there is some $C > 0$ such that $A \leq C\ B$. Also we will denote the $L^p(A)$-norm by $\vrac{\cdot}_{p,A}$. With $p'$ we denote the H\"older dual of $p$, $p' = \frac{p}{p-1}$.

\section{A regularity proof for \texorpdfstring{$n$}{n}-harmonic maps into spheres}\label{s:classical}
In this section we explain a scheme for proving two classical results: first for $n=2$ H{\'{e}}lein's \cite{Hel90} regularity for harmonic maps into spheres, and then for $n>2$ Strzelecki's/Fuchs' \cite{Strz94,FuchsnharmSphere} regularity for $n$-harmonic maps into spheres. The $n=2$-case uses ideas developed in \cite{DR1dSphere,SNHarmS10} which treated the fractional harmonic maps into spheres. The case $n>2$ follows from the proof of Theorem~\ref{th:mainS}. 

The arguments presented here are more lengthy and seem more complicated than the beautiful proofs by H{\'{e}}lein and Strzelecki. On the other hand, they are robust enough under disturbances, and work in particular with with non-local operators.

\subsection*{The two-dimensional case}
For $n = 2$, a critical point of the energy $E_1(v) = \int |\nabla v|^2$, $v: \Omega \to \S^{N-1}$, satisfies
\begin{equation}\label{eq:PDE2D}
\lap u = u\ |\nabla u|^2 \quad \mbox{in $\Omega$}.
\end{equation}

The goal is to show that there is a $\tau < 1$, on all (small) balls $B_{4r} \subset \Omega$,
\begin{equation}\label{eq:goal2D}
 \int_{B_r} |\nabla u|^{2} \leq \tau \int_{B_{4r}} |\nabla u|^{2} + \emph{good terms}.
\end{equation}
Indeed, if we assume \eqref{eq:goal2D} and the ``good terms'' behave as their name suggests, crucially using that $\tau < 1$ by iteration\footnote{For this kind of iteration, and the following Morrey and Campanato spaces, we refer to \cite[Chapter III]{GiaquintaMI}, and for Sobolev imbedding to \cite{Adams75}.} we obtain a $\sigma \in (0,1)$ such that for any small ball $B_\rho$
\begin{equation}\label{eq:afteriteration}
 \int_{B_\rho} |\nabla u|^{2} \leq \rho^\sigma (E_1(u)+C).
\end{equation}
The estimate \eqref{eq:afteriteration} tells us that $\nabla u$ belongs to the Morrey space $\mathcal{M}^{\sigma,2}$ strictly smaller than $L^2 \equiv \mathcal{M}^{0,2}$. Without \eqref{eq:goal2D} from $\nabla u \in L^2$ Sobolev embedding implied only $u \in BMO = \mathcal{L}^{n,2}$, where $\mathcal{L}^{\lambda,p}$ are the Campanato spaces; now from $\nabla u \in \mathcal{M}^{\sigma,2}$ via Sobolev imbedding we infer that $u \in \mathcal{L}^{4\frac{1-\sigma}{2-\sigma},2} = C^{0,2\frac{1-\sigma}{2-\sigma}}$. That is, \eqref{eq:goal2D} implies H\"older regularity of $u$.

In order to obtain \eqref{eq:goal2D}, one might be tempted to just multiply \eqref{eq:PDE2D} with $u$ (up to a cutoff-functions) and integrate by parts, but this implies only
\[
 \int_{B_r} |\nabla u|^{2} \leq \vrac{u}_{L^\infty(B_{4r})}\ \int_{B_{4r}} |\nabla u|^{2} + \emph{good terms}
\]
at best. So the main idea which was used in \cite{DR1dSphere,SNHarmS10} is to split up $\nabla u$: Since $|u| \equiv 1$, for any $x \in \Omega$, $\alpha=1,2$
\[
 |\partial_\alpha u(x)| \aleq |u(x) \cdot \partial_\alpha u(x)| + \max_{o \perp u(x), |o|=1} |o \cdot \partial_\alpha u(x)|.
\]
On the other hand, it is not too difficult to show, that
\[
 \max_{o \perp u(x), |o|=1} |o \cdot \partial_\alpha u(x)| \aleq \max_{\omega} |u^i \omega_{ij} \partial_\alpha u(x)|,
\]
where the maximum is over all finitely many $\omega \in \R^{n \times n}$, $\omega_{ij} = -\omega_{ji} \in \{-1,0,1\}$. This can be seen as a consequence of  Lagrange’s
identity, also the proof is given in the appendix of \cite{DSpalphSphere}. That is, 
\begin{equation}\label{eq:nablaudecomp}
 |\partial_\alpha u(x)| \aleq |u(x) \cdot \partial_\alpha u(x)| + \max_{\omega} |u^i \omega_{ij} \partial_\alpha u(x)|.
\end{equation}
Now,
\[
 u(x) \cdot \partial_\alpha u(x) = \frac{1}{2} \partial_\alpha |u(x)|^2 = \frac{1}{2} \partial_\alpha 1 = 0.
\]
Thus, in order to obtain \eqref{eq:goal2D}, we need to show for an arbitrary but constant $\omega \in \R^{n \times n}$, $\omega_{ij} = -\omega_{ji} \in \{-1,0,1\}$,
\begin{equation}\label{eq:goal2Dtilde}
  \int_{B_r} |u^i \omega_{ij} \nabla u^j|^{2} \leq \tau \int_{B_{4r}} |\nabla u|^{2} + \emph{good terms}.
\end{equation}
But now, \eqref{eq:PDE2D} together with the antisymmetry of $\omega$ implies $\operatorname{div}(u^i \omega_{ij} \nabla u^j)  = 0$, and 
\[
 \operatorname{curl} (u^i \omega_{ij} \nabla u^j) = \omega_{ij}\ \nabla^\perp u^i \ \nabla u^j,
\]
where $\nabla^\perp = (-\partial_2,\partial_1)^T$. Since $\omega$ is constant, the right-hand side is a product of divergence-free and rotation-free vectorfields and by \cite{CLMS} belongs to the Hardy space. In particular it can be tested against BMO-functions (such as $u$), and by a Hodge-decomposition this implies
\[
  \int_{B_r} |u^i \omega_{ij} \nabla u^j|^{2} \leq C\ [u]_{BMO,B_{4r}} \int_{B_{4r}} |\nabla u|^{2} + \emph{good terms}.
\]
Now in contrast to $\vrac{u}_{L^\infty(B_{4r})}$, for small enough radii $r$ we know that $[u]_{BMO,B_{4r}}$ is small from which we can construct $\tau < 1$ in \eqref{eq:goal2Dtilde} and thus in particular we have \eqref{eq:goal2D}, which implies H\"older continuity.

\subsection*{The general case}
If $n > 2$, the Euler-Lagrange equations \eqref{eq:PDE2D} become
\begin{equation}\label{eq:PDEnD}
\operatorname{div}(|\nabla u|^{n-2}\nabla u) = u\ |\nabla u|^n \quad \mbox{in $\Omega$}.
\end{equation}
Following the rough idea of the 2D-case, one wants to show now that there is a $\tau < 1$, and on all (small) balls $B_{4r} \subset \Omega$,
\begin{equation}\label{eq:goalnDFirstTry}
 \int_{B_r} |\nabla u|^{n} \leq \tau \int_{B_{4r}} |\nabla u|^{n} + \emph{good terms},
\end{equation}
which again implies immediately H\"older regularity of $u$.

Now, one would decompose $|\nabla u|$ as in \eqref{eq:nablaudecomp}, and we would try show the existence of $\tau < 1$ such that 
\[
   \int_{B_r} \left ||\nabla u|^{n-2}\ u^i \omega_{ij} \nabla u^j \right|^{n'} \leq \tau \int_{B_{4r}} |\nabla u|^{n} + \emph{good terms}.
\]
In order to do that, one would compute that the divergence
\[
 \operatorname{div}(|\nabla u|^{n-2}\ u^i \omega_{ij} \nabla u^j) = 0,
\]
but when one computes the rotation a problem occurs, since in general for $n \neq 2$ it is not clear why it should 	be true that
\[
 \operatorname{curl}(|\nabla u|^{n-2}\ u^i \omega_{ij} \nabla u^j) \in \mathcal{H}^1.
\]
We conclude that it does not seem feasible to decompose $|\nabla u|$ in that way. 

Instead, one first shows, see Lemma~\ref{la:lefthandside}, that for the $\alpha$-th Riesz transform $\Rz_\alpha$,
\begin{equation}\label{eq:nablaubyRz}
\vrac{\nabla u}_{n,B_r}^n \aleq \int_{B_{2r}} \left |\Rz_\alpha [|\nabla u|^{n-2} \partial_\alpha u ]\right|^{n'} + \emph{good terms}.
\end{equation}
This is the crucial point that makes our argument work: trying to estimate $|\nabla u|^{n-2} \nabla u$ seems to inevitably lead to a rotation-problem, like the one that spoiled our first attempt above. Instead, we estimate $\Rz[|\nabla u|^{n-2} \nabla u]$, a term which looks more complicated, but where this rotation-problem does not appear:

We decompose as in \eqref{eq:nablaudecomp},
\begin{align}
 |\Rz_\alpha [|\nabla u|^{n-2}\ \partial_\alpha u ]| \aleq\ &|u^i  \Rz_\alpha [|\nabla u|^{n-2}\ \partial_\alpha u^i ]| \label{eq:thedecomp}\\
 &+ \max_{\omega} |u^i  \omega_{ij} \Rz_\alpha [|\nabla u|^{n-2}\ \partial_\alpha u^j ]|\nonumber.
\end{align}
We shall call the first term the orthogonal part (since $u(x)$ is orthogonal to the tangential space $T_{u(x)} \S^{N-1}$). It is now non-zero (in contrast to the $n=2$- case), but $u^i \partial_\alpha u^i = 0$ still implies
\begin{equation}\label{eq:uRzpartialuclassest}
 u^i  \Rz_\alpha [|\nabla u|^{n-2}\ \partial_\alpha u^i ] = \mathcal{C}(u^i,\Rz_\alpha)[|\nabla u|^{n-2}\ \partial_\alpha u^i],
\end{equation}
with the commutator
\begin{equation}\label{eq:def:CRWcommutator}
 \mathcal{C}(b,T)[v] = bT[v] - T[bv].
\end{equation}
Now employing the Coifman-Rochberg-Weiss Theorem \cite{CRW76}, we obtain
\begin{align*}
 &\vrac{u^i  \Rz_\alpha [|\nabla u|^{n-2}\ \partial_\alpha u^i ]}_{L^{n'}(B_{2r})} \\
 &\aleq [u]_{BMO,B_{3r}}\ \vrac{|\nabla u|^{n-2}\ \partial_\alpha u^i}_{L^{n'}(B_{3r})} + \emph{good terms}\\
 &\aleq [u]_{BMO,B_{3r}}\ \vrac{\nabla u}_{L^{n}(B_{3r})}^{n-1} + \emph{good terms}.
\end{align*}
The ``good terms'' stem from cut-off functions, and for the precise formulation we refer to Lemma~\ref{la:mainl}. Now, we can use that $[u]_{BMO}$ is small on small sets, and have for the orthogonal part, that for some $\tau < 1$,
\begin{equation}\label{eq:orthogonalpart}
 \vrac{u^i  \Rz_\alpha [|\nabla u|^{n-2}\ \partial_\alpha u^i ]}_{L^{n'}(B_{2r})} \\
 \leq  \frac{\tau}{2} \vrac{\nabla u}_{L^{n}(B_{3r})}^{n-1} + \emph{good terms}.
\end{equation}
It remains to estimate the ``tangential part'', and here we use the equation \eqref{eq:PDEnD}. Firstly, by Lemma~\ref{la:EstimatebyPDE}, for some $\varphi \in C_0^\infty(B_{2r})$, $\vrac{\nabla \varphi}_{n} \leq 1$,
\begin{align}
&\vrac{u^i  \omega_{ij} \Rz_\alpha [|\nabla u|^{n-2}\ \partial_\alpha u^j ]}_{n',B_{2r}} \label{eq:estimatebyPDEarg}\\
\aleq\quad  &\int_{\R^n} u^i  \omega_{ij} \Rz_\alpha [|\nabla u|^{n-2}\ \partial_\alpha u^j ]\ \laps{1} \varphi + \emph{good terms} \nonumber.
\end{align}
Now since $\Rz_\alpha \laps{1} = \partial_\alpha$, and using \eqref{eq:PDEnD},
\[
  \laps{1} \Rz_\alpha [|\nabla u|^{n-2}\ \partial_\alpha u^j ] = \operatorname{div}(|\nabla u|^{n-2} \nabla u^j)=|\nabla u|^n u^j,
\]
and consequently,
\begin{equation}\label{eq:nELapplied}
  \int_{\R^n} \varphi\ u^i  \omega_{ij} \laps{1} \Rz_\alpha [|\nabla u|^{n-2}\ \partial_\alpha u^j ] = 0.
\end{equation}
Moreover, by the antisymmetry of $\omega$,
\begin{equation}\label{eq:secondguyPDE}
  \omega_{ij} \Rz_\alpha [\laps{1} u]\ |\nabla u|^{n-2}\ \partial_\alpha u^j  \equiv 0.
\end{equation}
We recall the bi-commutator $H_s$, which measures how much $\laps{s}$ is away from having a product rule,
\begin{equation}\label{eq:defHs}
 H_{s}(a,b) = \laps{s} (ab) - a \laps{s}b - b \laps{s} a.
\end{equation}
From \eqref{eq:estimatebyPDEarg} we then have
\begin{align*}
&\vrac{u^i  \omega_{ij} \Rz_\alpha [|\nabla u|^{n-2}\ \partial_\alpha u^j ]}_{n',B_{2r}}\\
\overset{\eqref{eq:nELapplied}}{\aleq}\quad  &-\omega_{ij}  \int_{\R^n} H_1 (u^i,\varphi)\ \Rz_\alpha [|\nabla u|^{n-2}\ \partial_\alpha u^j ]\\
& - \omega_{ij}  \int_{\R^n} \laps{1} u^i \ \Rz_\alpha [|\nabla u|^{n-2}\ \partial_\alpha u^j ]\ \varphi\\
&+ \emph{good terms}\\
\overset{\eqref{eq:secondguyPDE}}{=}\quad  &\omega_{ij}  \int_{\R^n} H_1 (u^i,\varphi) \ \Rz_\alpha [|\nabla u|^{n-2}\ \partial_\alpha u^j ]\\
& + \omega_{ij}  \int_{\R^n} \laps{1} u^i \ \mathcal{C}(\Rz_\alpha,\varphi) [|\nabla u|^{n-2}\ \partial_\alpha u^j ]\\
&+ \emph{good terms}
\end{align*}
where $\mathcal{C}(\cdot,\cdot)[\cdot]$ is again the commutator defined in \eqref{eq:def:CRWcommutator}. The estimates on bi-commutators established in \cite{DR1dSphere,Sfracenergy} and the commutator estimates in \cite{CRW76}, imply 
\begin{align*}
&\vrac{u^i  \omega_{ij} \Rz_\alpha [|\nabla u|^{n-2}\ \partial_\alpha u^j ]}_{n',B_{2r}}\\
\aleq &\quad  \vrac{\laps{1} u}_{n,B_{4r}}\ \vrac{\nabla u}^{n-1}_{n,B_{4r}}\ \brac{\vrac{\laps{1} \varphi}_{n} + [\varphi]_{BMO}} 	+ \emph{good terms}.
\end{align*}
Since $\nabla u \in L^n$, so is $\laps{1} u$, and for small radii, we have
\[
\vrac{u^i  \omega_{ij} \Rz_\alpha [|\nabla u|^{n-2}\ \partial_\alpha u^j ]}_{n',B_{2r}} \leq \frac{\tau}{2} \vrac{\nabla u}^{n-1}_{n,B_{4r}} 	+ \emph{good terms}.
\]
This, \eqref{eq:orthogonalpart} and the decomposition \eqref{eq:thedecomp}, imply \eqref{eq:goalnDFirstTry}, which again implies regularity. \qed

%

\section{The Goal}\label{sec:notaton}
The analogue of the ``goal'' \eqref{eq:goal2D}, \eqref{eq:goalnDFirstTry} is the following Lemma, for which we need to introduce some notation: We fix some reference ball scale $B_R(x_0) \subset \Omega$. If $\chi_A$ is the characteristic function on $A$, we denote for $l \in \Z$,
\begin{equation}\label{eq:cutdef}
 \cut_{l} := \chi_{B_{2^lR}(x_0)}, \mbox{ and } \cutA_l := \cut_l - \cut_{l-1}.
\end{equation}

\begin{lemma}\label{la:mainl}
Let $p_s := \frac{n}{s}$. Assume that $u$ as in Theorem~\ref{th:mainS} or Theorem~\ref{th:mainSO}. Then for some $L_0 \in \Z$, $K_0 \in \mathbb{N}$, $\tau \in (0,1)$, $\sigma > 0$, for any $L \leq L_0$, $K \geq K_0$,
\[
 \vrac{\cut_{L} \laps{s} u}_{p_s}^{p_s} \leq \tau\ \|\cut_{L+K} \laps{s} u\|_{p_s}^{p_s} + C\ 2^{-K\sigma}\ \sum_{l=1}^\infty  2^{-l\sigma}\ \|\cut_{L+K+l}\laps{s}u \|_{p_s}^{p_s}.\\
\]
\end{lemma}

This Lemma implies Theorem~\ref{th:mainS}, Theorem~\ref{th:mainSO} by iteration and Sobolev imbedding, essentially as described in Section~\ref{s:classical}. For the details, we refer to, e.g., \cite{DR1dSphere} and also to the appendix in \cite{BPSknot12}.
\begin{proof}[Proof of Lemma~\ref{la:mainl}]
Before starting the proof, we set
\[
 \mbox{``good terms $(C,K,\sigma)$''} := C\ 2^{-K\sigma}\ \sum_{l=1}^\infty  2^{-l\sigma}\ \|\cut_{L+K+l}\laps{s}u \|_{p_s}^{p_s}.
\]
Note that for any $\eps > 0$ there is $K$ sufficiently large such that for any $\tilde{K} > K$,
\begin{align*}
 &\mbox{``good terms $(C,K,\sigma)$''}\\
 \leq\ &\eps \|\cut_{L+\tilde{K}}\laps{s}u \|_{p_s}^{p_s} + C\ 2^{-\tilde{K}\ \frac{K}{\tilde{K}}\sigma}\ \sum_{l=1}^\infty  2^{-l\sigma}\ \|\cut_{L+\tilde{K}+l}\laps{s}u \|_{p_s}^{p_s}\\
 =\ &\eps \|\cut_{L+\tilde{K}}\laps{s}u \|_{p_s}^{p_s} + \mbox{``good terms $(C,\tilde{K},\tilde{\sigma})$''},
\end{align*}
where $\tilde{\sigma} = \frac{K}{\tilde{K}} \sigma$. Consequently, for proving the claim, we don't need to care to much about the precise values of $K$, $C$, $\sigma$, as long as $K$ is sufficiently large.

Now we start the proof: Fix $\eps > 0$ to be determined later. By Lemma~\ref{la:lefthandside}, for some $\tau \in (0,1)$,
\begin{align}
 \vrac{\cut_{L}\laps{s} u}_{p_s}^{p_s}\nonumber &\leq 
 \tau\ \|\cut_{L+K}\laps{s} u\|_{p_s}^{p_s} + \mbox{``good terms''}\nonumber\\
 &\quad + C\ \|\cut_{L+K} \Rz_\beta [|\nabla^s u|^{p_s-2}\ \partial_\beta^s u] \|_{p_s'}^{p_s'}\label{eq:beforedecomp}.
\end{align}
Let us first concentrate on the sphere case: we use the decomposition \eqref{eq:thedecomp}, which is valid as long as $\supp \chi_{L+K} \subset \Omega$.
\[
 \|\cut_{L+K} \Rz_\beta [|\nabla^s u|^{p_s-2}\ \partial_\beta^s u] \|_{p_s'}^{p_s'}
\]
\[
 \aleq \|u^i \cut_{L+K} \Rz_\beta [|\nabla^s u|^{p_s-2}\ \partial_\beta^s u^i] \|_{p_s'}^{p_s'} + \max_{\omega} \|u^j \omega_{ij} \cut_{L+K} \Rz_\beta [|\nabla^s u|^{p_s-2}\ \partial_\beta^s u^j] \|_{p_s'}^{p_s'} 
\]
For the first term we use Lemma~\ref{la:orthogonalpart},
\[
\|u^i \cut_{L+K} \Rz_\beta [|\nabla^s u|^{p_s-2}\ \partial_\beta^s u^i] \|_{p_s'}^{p_s'}
\]
\[\aleq
\brac{\vrac{\cut_{L+2K} \laps{s} u}_{p_s} + 2^{-K\sigma}}^{p_s'}\ \vrac{\cut_{L+2K} \laps{s} u}_{p_s}^{p_s}.
\]
If $K \in \N$ is large enough, and $L \in \Z$ negative enough so that $\vrac{\chi_{L+2K} \laps{s} u}_{p_s}^{p_s'} < \eps$ by absolute continuity of the integral, and $2^{-K\sigma p_s'} < \eps$, we arrive at
 \begin{align*}
 \vrac{\cut_{L}\laps{s} u}_{p_s}^{p_s} &\leq 
 (\tau +C \eps)\ \|\cut_{L+2K}\laps{s} u\|_{p_s}^{p_s} + \mbox{``good terms''}\\
 &\quad + \max_{\omega} \|u^j \omega_{ij} \cut_{L+K} \Rz_\beta [|\nabla^s u|^{p_s-2}\ \partial_\beta^s u^j] \|_{p_s'}^{p_s'}.
\end{align*}
For the second term, we first use Lemma~\ref{la:EstimatebyPDE} and then Lemma~\ref{la:tangentialpart}, to obtain
\[
 \|u^j \omega_{ij} \cut_{L+K} \Rz_\beta [|\nabla^s u|^{p_s-2}\ \partial_\beta^s u^j] \|_{p_s'}^{p_s'} \aleq (\vrac{\cut_{L+2K}\laps{s} u}_{p_s}^{p_s'} + 2^{-\sigma K}) \vrac{\cut_{L+K}\laps{s} u}^{p_s}_{p_s} + \mbox{``good terms''}.
\]
Again, if $K$ is large enough $L+2K$ is small enough, this implies
\[
 \|u^j \omega_{ij} \cut_{L+K} \Rz_\beta [|\nabla^s u|^{p_s-2}\ \partial_\beta^s u^j] \|_{p_s'}^{p_s'} \aleq \eps\ \vrac{\cut_{L+K}\laps{s} u}^{p_s}_{p_s} + \mbox{``good terms''}.
\]
Thus, we finally arrive at
\begin{align*}
 \vrac{\cut_{L}\laps{s} u}_{p_s}^{p_s} &\leq 
 (\tau +C \eps)\ \|\cut_{L+2K}\laps{s} u\|_{p_s}^{p_s} + \mbox{``good terms''}.
\end{align*}

Since $C$ is a generic constant, depending on the dimension, $s$, and possibly $\vrac{\laps{s} u}_{p_s}$, if $\eps$ was chosen small enough such that $\tilde{\tau} := \tau + C\eps < 1$ this gives the claim with $\tilde{\tau}$ instead of $\tau$.

In the case of $SO(N)$, in \eqref{eq:beforedecomp} we use decomposition Proposition~\ref{pr:decompSO}. Then, for symmetric $\theta_{ij} = \theta_{ji} \in \{-1,0,1\}$ (i.e. the orthogonal part), we apply Lemma~\ref{la:orthogonalpartSO} to
\[
 \|\cut_{L+K} \theta_{ij} u_{ik} \Rz_\beta [|\nabla^s u|^{p_s-2}\ \partial_\beta^s u_{kj}] \|_{p_s'}^{p_s'},
\]
and Lemma~\ref{la:tangentialpartSO} to antisymmetric $\omega_{ij} = \omega_{ji} \in \{-1,0,1\}$ (i.e. the tangential part) to
\[
 \|\cut_{L+K} \omega_{ij} u_{ik} \Rz_\beta [|\nabla^s u|^{p_s-2}\ \partial_\beta^s u_{kj}] \|_{p_s'}^{p_s'},
\]
and obtain exactly the same estimates.
\end{proof}

\section{Estimates on the Left-hand side}
Here, we prove the estimate that leads to \eqref{eq:nablaubyRz}. It is an extension and localization of the following simple argument
\begin{align*}
 \vrac{\nabla u}_{n}^n &= \int |\nabla u|^{n-2} \nabla u\cdot \nabla u = \int |\nabla u|^{n-2} \partial_\alpha u\ \Rz_\alpha \laps{1} u\\
 &=- \int \Rz_\alpha [|\nabla u|^{n-2} \partial_\alpha u]\ \laps{1} u \aleq \vrac{\Rz_\alpha [|\nabla u|^{n-2} \partial_\alpha u]}_{n'}\ \vrac{\laps{1} u}_{n'}.
\end{align*}
Now one uses that $\vrac{\laps{1} u}_{p} \approx \vrac{\nabla u}_p$, and obtains an estimate similar to \eqref{eq:nablaubyRz}. 

Although this arguments seems trivial, it is related to one of the main problems in the $n$-harmonic map case, or related $n$-laplace PDE's, e.g. such as \cite{DM10}:

By the Iwaniec stability result \cite{Iw92}, one can prove an estimate like \eqref{eq:nablaubyRz} also for $p \approx n$, i.e.
\begin{align*}
 \vrac{\nabla u}_{p}^{p} \aleq \vrac{\Rz_\alpha [|\nabla u|^{n-2} \partial_\alpha u]}^{\frac{p}{p-1}}_{\frac{p}{2p-n-1}},
\end{align*}
However, to our knowledge, it is not know whether
\[
 \vrac{\nabla u}_{(n,\infty)}^{n} \aleq \vrac{\Rz_\alpha [|\nabla u|^{n-2} \partial_\alpha u]}^{n'}_{(n',\infty)},
\]
where $L^{(n,\infty)}$ is the weak $L^n$-space. For $n=2$ this estimate is elementary, and it is important for some of the proofs of harmonic maps.
  
\begin{lemma}[Left-Hand side]\label{la:lefthandside}
Recall \eqref{eq:cutdef}. There exists a constant $T_1$, a $\tau \in (0,1)$, $\sigma > 0$, such for any $T \geq T_1$, and any $u$,
\begin{align*}
 \vrac{\cut_{S}\laps{s} u}_{p_s}^{p_s} &\leq 
 C\ \|\cut_{S+T} \Rz_\beta [|\nabla^s u|^{p_s-2}\ \partial_\beta^s u] \|_{p_s'}^{p_s'} \\
 & \quad + \tau\ \|\cut_{S+T}\laps{s} u\|_{p_s}^{p_s}\\
%
&\quad + C\ 2^{-T\sigma}\ \sum_{l=1}^\infty  2^{-l\sigma}\ \|\cut_{S+T+l}\laps{s}u \|_{p_s}^{p_s}.\\
%
\end{align*}
\end{lemma}
\begin{proof}
In order to reduce the number of indices a little bit, we assume $S = 0$. This is fine, since all the arguments in the following can be shifted by $+S$. We follow the argument presented at the beginning of this section, only taking care about localization. Then we have,
\begin{align*}
 \vrac{\cut_0\ \nabla^s u}_{p_s}^{p_s} 
 &= \int_{\R^n} |\nabla^s u|^{p_s-2}\ \partial_\beta^s u^i\cdot \Rz_\beta [\cut_{K}\laps{s} u^i]\\
 &\quad - \int (\cut_{K+L} - \cut_0)\ |\nabla^s u|^{p_s-2}\ \partial_\beta^s u^i\cdot \Rz_\beta [\cut_{K}\laps{s} u^i]\\
 &\quad - \sum_{l=K+L+1}^\infty \int \cutA_{l}\ |\nabla^s u|^{p_s-2}\ \partial_\beta^s u^i\cdot \Rz_\beta [\cut_{K}\laps{s} u^i]\\
 &\quad + \sum_{k =K+1}^\infty \int_{\R^n} \cut_0 |\nabla^s u|^{p_s-2}\ \partial_\beta^s u^i\cdot \Rz_\beta [\cutA_{k}\laps{s} u^i]\\
 &=: I - II - III + IV.
\end{align*}
The first term $I$ is exactly the part we want. Young's inequality and partial integration of $\Rz_\beta$ implies for any $\eps > 0$,
\[
 I \leq \frac{C}{\eps^{p_s'}} \|\cut_K\  \Rz_\beta [|\nabla^s u|^{p_s-2}\ \partial_\beta^s u^i] \|_{p_s'}^{p_s'} +  \eps^{p_s} \| \cut_K\  \laps{s} u^i \|_{p_s}^{p_s}.
\]
Next for any $\delta > 0$,
\[
II \leq \frac{C}{\delta^{p_s'}} \vrac{(\cut_{K+L} - \cut_0) |\nabla^s u|}_{p_s}^{p_s} + \delta^{p_s} \vrac{\cut_K\ \laps{s} u^i}_{p_s}^{p_s}.
\]
Then we apply Lemma~\ref{la:lapstoRiesz}, and have for any $M \geq 2$,
\begin{align*}
II 
 &\leq \frac{C}{\delta^{p_s'}} \|(\cut_{K-L+M}-\cut_{-M})\ \laps{s} u \|_{p_s}^{p_s} \\
 & \quad + \frac{C}{\delta^{p_s'}} \sum_{k=1}^\infty 2^{-(M+k)\frac{n}{p}}\ \|\cutA_{K-L+M+k}\ \laps{s} u\|_{p_s}^{p_s} \\
 & \quad + 2^{(K-L)n-Mn(p_s-1)}\ \frac{C}{\delta^{p_s'}}\ \|\cut_{-M}\ \laps{s} u\|_{p_s}^{p_s}\\
 & + \delta^{p_s} \vrac{\cut_K\ \laps{s} u^i}_{p_s}^{p_s}.
\end{align*}
Now we treat $III$. Applying Lemma~\ref{la:RieszlapsII}, and Lemma~\ref{la:local},
\begin{align*}	
 III & \aleq \sum_{l=K+L+1}^\infty \vrac{\cutA_{l} \nabla^s u}_{p_s}^{p_s-1} \vrac{\cutA_{l}\Rz_\beta [\cut_{K}\laps{s} u^i]}_{p_s}\\
 & \aleq \sum_{l=K+L+1}^\infty \left (\|\cut_{l+M} \laps{s}u \|_{p_s} +  \sum_{k=1}^\infty 2^{-\frac{n}{p}(M+k)} \|\cutA_{l+M+k} \laps{s} u \|_{p_s} \right )^{p_s-1}\ 2^{(K-l)\frac{n}{p_s'}}\ \vrac{\cut_{K}\laps{s} u^i}_{p_s}\\
 & \aleq \sum_{l=K+L+1}^\infty \left (\|\cut_{l+M} \laps{s}u \|_{p_s}^{p_s-1} +  2^{-M\frac{n}{p'}}\sum_{k=1}^\infty 2^{-\frac{n}{p}(k)\min\{(p_{s}-1),1\}} \|\cutA_{l+M+k} \laps{s} u \|_{p_s}^{p_s-1} \right )\ 2^{(K-l)\frac{n}{p_s'}}\ \vrac{\cut_{K}\laps{s} u^i}_{p_s}\\
 &\aleq \sum_{l=K+L+1}^\infty  2^{(K-l)\frac{n}{p_s'}}\ \|\cut_{l+M} \laps{s}u \|_{p_s}^{p_s}\\
 & \quad + \begin{cases} 
 \sum_{i=1}^\infty  \ 2^{-\frac{n}{p_s'}(L+M+1+i)}\  
  |i|\ \|\cutA_{K+L+M+1+i} \laps{s} u \|_{p_s}^{p_s}  \quad &p_s \leq 2,\\
 \sum_{i=1}^\infty  2^{-\frac{n}{p_s'}M} \ 2^{\frac{n}{p_s'}K}\ 2^{-\frac{n}{p_s}i}\  
  \ \|\cutA_{K+L+M+1+i} \laps{s} u \|_{p_s}^{p_s}  \quad &p_s > 2.
	   \end{cases}
 \end{align*}
Finally, for any $\gamma > 0$,
\begin{align*}
 IV
 \aleq &\vrac{\cut_0 \nabla^s u}_{p_s}^{p_s-1}   \sum_{k =K+1}^\infty\ 2^{-k\frac{n}{p}}\ \vrac{\cut_{k}\laps{s} u^i}_{p_s}\\
 \aleq &\gamma^{p_s'} \vrac{\cut_0 \nabla^s u}_{p_s}^{p_s} + \gamma^{-p_s} \ 2^{-K\frac{n}{p_s'}}\ \sum_{k =K+1}^\infty\ 2^{-k\frac{n}{p_s}}\ \vrac{\cut_{k}\laps{s} u}_{p_s}^{p_s}.
\end{align*}
Together, we arrive for some constant $C$ depending only on $p_s$ and the dimensions, for any $\eps,\delta,\gamma \in (0,1)$, and for any $K,L,M \in \mathbb{N}$, and $\geq 2$
\[
\vrac{\cut_0\ \nabla^s u}_{p_s}^{p_s} \leq 
\]
\[
\frac{C}{\eps^{p_s'}} \|\cut_K\  \Rz_\beta [|\nabla^s u|^{p_s-2}\ \partial_\beta^s u^i] \|_{p_s'}^{p_s'} +  \eps^{p_s} \| \cut_K\  \laps{s} u^i \|_{p_s}^{p_s}
\]
\[
+\frac{C}{\delta^{p_s'}} \Big (\|(\cut_{K-L+M}-\cut_{-M})\ \laps{s} u \|_{p_s}^{p_s}
\]
\[
 + \sum_{k=1}^\infty 2^{-(M+k)\frac{n}{p}}\ \|\cutA_{K-L+M+k}\ \laps{s} u\|_{p_s}^{p_s} \\
\]
\[
 + 2^{(K-L)n-Mn(p_s-1)}\  \|\cut_{-M}\ \laps{s} u\|_{p_s}^{p_s}\\
 \]
 \[
 + \delta^{p_s} \vrac{\cut_K\ \laps{s} u^i}_{p_s}^{p_s}\\
 \]
\[
+C\sum_{l=K+L+1}^\infty  2^{(K-l)\frac{n}{p_s'}}\ \|\cut_{l+M} \laps{s}u \|_{p_s}^{p_s}
\]
\[
+C\begin{cases} 
 \sum_{i=1}^\infty  \ 2^{-\frac{n}{p_s'}(L+M+1+i)}\  
  |i|\ \|\cutA_{K+L+M+1+i} \laps{s} u \|_{p_s}^{p_s}  \quad &p_s \leq 2,\\
 \sum_{i=1}^\infty  2^{-\frac{n}{p_s'}M} \ 2^{\frac{n}{p_s'}K}\ 2^{-\frac{n}{p_s}i}\  
  \ \|\cutA_{K+L+M+1+i} \laps{s} u \|_{p_s}^{p_s}  \quad &p_s > 2.
	   \end{cases}\\
\]
\[
+C\gamma^{p_s'} \vrac{\cut_0 \nabla^s u}_{p_s}^{p_s} + C\gamma^{-p_s} \ 2^{-K\frac{n}{p_s'}}\ \sum_{k =K+1}^\infty\ 2^{-k\frac{n}{p_s}}\ \vrac{\cut_{k}\laps{s} u^i}_{p_s}^{p_s}.
\]

Taking $\gamma > 0$ small enough, we can absorb the $\gamma^{p_s'}$-term. Next, take $M := 2K$, $L =2K$. Then,
\begin{align*}
\vrac{\cut_0\ \nabla^s u}_{p_s}^{p_s} &\leq 
\frac{C}{\eps^{p_s'}} \|\cut_K\  \Rz_\beta [|\nabla^s u|^{p_s-2}\ \partial_\beta^s u^i] \|_{p_s'}^{p_s'} \\
&\quad +\frac{C}{\delta^{p_s'}} \ \|(\cut_{K}-\cut_{-2K})\ \laps{s} u \|_{p_s}^{p_s} \\
 & \quad + C\left (\frac{1}{\delta^{p_s'}}  2^{-K(n+2p_s-2)}\ +\delta^{p_s} +\eps^{p_s} \right )  \|\cut_{K}\ \laps{s} u\|_{p_s}^{p_s}\\
%
&\quad +C\ \left (1+\delta^{-p_s'}+\gamma^{-p_s}\right )\ 2^{-K\sigma_1}\ \sum_{l=1}^\infty  2^{-l\sigma_2}\ \|\cut_{5K+l} \laps{s}u \|_{p_s}^{p_s}\\
\end{align*}
Now, for some $S \in \mathbb{N}$ we convert the left-hand side into a $\laps{s}$-term: By Lemma~\ref{la:lapstoRieszII}
\begin{align*}
 \vrac{\cut_{-S} \laps{s} u}_{p_s} &\aleq \|\cut_0 \nabla^s u\|_{p} + \sum_{l=1}^\infty l\ 2^{-\frac{n}{p}(S+l)}\|\cutA_{l} \laps{s} u \|_{p_s}\\
\end{align*}
Together with the above estimate we arrive at,
\begin{align*}
\vrac{\cut_{-S} \laps{s} u}_{p_s}^{p_s} &\leq 
\frac{C}{\eps^{p_s'}} \|\cut_K\  \Rz_\beta [|\nabla^s u|^{p_s-2}\ \partial_\beta^s u^i] \|_{p_s'}^{p_s'} \\
&\quad +\frac{C}{\delta^{p_s'}} \ \|(\cut_{K}-\cut_{-2K})\ \laps{s} u \|_{p_s}^{p_s} \\
 & \quad + C\left (\frac{1}{\delta^{p_s'}}  2^{-K\sigma_1}\ +\delta^{p_s} +\eps^{p_s} \right )  \|\cut_{K}\ \laps{s} u\|_{p_s}^{p_s}\\
%
&\quad +C\ (\left (1+\delta^{-p_s'}+\gamma^{-p_s}\right )\ 2^{-K\sigma_1}+2^{-S\sigma_1})\ \sum_{l=1}^\infty  2^{-l\sigma_2}\ \|\cut_{5K+l} \laps{s}u \|_{p_s}^{p_s}.\\
%
\end{align*}
The proof of the above estimate did not depend at all on the scale, so we can shift all indices by $+S$. Setting $S:=2K$, for possibly smaller $\sigma_1$, $\sigma_2$, for any $T \in \mathbb{N}$, $T \geq T_0$ for some uniform $T_0 > 0$, we have shown
\begin{align*}
\vrac{\cut_{0} \laps{s} u}_{p_s}^{p_s} &\leq 
\frac{C}{\eps^{p_s'}} \|\cut_{T}\  \Rz_\beta [|\nabla^s u|^{p_s-2}\ \partial_\beta^s u^i] \|_{p_s'}^{p_s'} \\
&\quad +\frac{C}{\delta^{p_s'}} \ \|(\cut_{T}-\cut_{0})\ \laps{s} u \|_{p_s}^{p_s} \\
 & \quad + C\left (\frac{1}{\delta^{p_s'}}  2^{-T\sigma_1}\ +\delta^{p_s} +\eps^{p_s} \right )  \|\cut_{T}\ \laps{s} u\|_{p_s}^{p_s}\\
%
&\quad +C\ \left (1+\delta^{-p_s'}+\gamma^{-p_s}\right )\ 2^{-T\sigma_1}\ \sum_{l=1}^\infty  2^{-l\sigma_2}\ \|\cut_{T+l} \laps{s}u \|_{p_s}^{p_s}.\\
%
\end{align*}
Adding $\frac{C}{\delta^{p_s'}} \ \|\cut_{0}\ \laps{s} u \|_{p_s}^{p_s}$ on both sides and dividing by $(1+\frac{C}{\delta^{p_s'}})$, we have
\begin{align*}
\vrac{\cut_{0} \laps{s} u}_{p_s}^{p_s} &\leq 
(1+\frac{C}{\delta^{p_s'}})^{-1} \frac{C}{\eps^{p_s'}} \|\cut_{T}\  \Rz_\beta [|\nabla^s u|^{p_s-2}\ \partial_\beta^s u^i] \|_{p_s'}^{p_s'} \\
 & \quad + (1+\frac{C}{\delta^{p_s'}})^{-1} \left (\frac{C}{\delta^{p_s'}} + C\left (\frac{1}{\delta^{p_s'}}  2^{-T\sigma_1}\ +\delta^{p_s} +\eps^{p_s} \right )  \right)\ \|\cut_{T}\ \laps{s} u\|_{p_s}^{p_s}\\
%
&\quad +(1+\frac{C}{\delta^{p_s'}})^{-1} C\ \left (1+\delta^{-p_s'}+\gamma^{-p_s}\right )\ 2^{-T\sigma_1}\ \sum_{l=1}^\infty  2^{-l\sigma_2}\ \|\cut_{T+l} \laps{s}u \|_{p_s}^{p_s}.\\
%
\end{align*}
Now we choose $\eps$, $\delta$, and then $T_1 \geq T_0$ such that for any $T \geq T_1$,
\[
 C\left (\frac{1}{\delta^{p_s'}}  2^{-T\sigma_1}\ +\delta^{p_s} +\eps^{p_s} \right ) < 1.
\]
Then for some $\tau < 1$, for any $T \geq T_1$,
\begin{align*}
\vrac{\cut_{0} \laps{s} u}_{p_s}^{p_s} &\leq 
\tilde{C}\ \|\cut_{T}\  \Rz_\beta [|\nabla^s u|^{p_s-2}\ \partial_\beta^s u^i] \|_{p_s'}^{p_s'} \\
 & \quad + \tau\ \|\cut_{T}\ \laps{s} u\|_{p_s}^{p_s}\\
%
&\quad +\tilde{C} \ 2^{-T\sigma_1}\ \sum_{l=1}^\infty  2^{-l\sigma_2}\ \|\cut_{T+l} \laps{s}u \|_{p_s}^{p_s}.\\
%
\end{align*}
\end{proof}

\section{Right-hand side I: orthogonal part}
The following lemma is the fractional analogue of the argument in \eqref{eq:uRzpartialuclassest}:
\begin{lemma}\label{la:orthogonalpart}
For any $K \in \Z$, $L \in \N$, $L \geq L_0$, and any $u$ the following holds:
If $\chi_{2L} |u| \equiv 1$, then for some uniform $\sigma > 0$,
\begin{align*}
&\|\cut_{K} u^i\ \Rz_\beta [|\nabla^s u|^{p_s-2}\ \partial_\beta^s u^i] \|_{p_s'}\\
\aleq& \brac{\vrac{\cut_{K+L} \laps{s} u}_{p_s} + 2^{-L\sigma}}\ \vrac{\cut_{K+L} \laps{s} u}_{p_s}^{p_s-1}\\
&+ (1+\vrac{\laps{s} u}_{p_s})\ \sum_{k=1}^\infty 2^{-(L+k)\sigma} \vrac{\cut_{K+L+k} \laps{s} u}_{p_s}^{p_s-1}.
\end{align*}
\end{lemma}
\begin{proof}
Again, we may assume $K = 0$, in order to work with less indices.

For any constant $c^i$, we again decompose the quantity in question into differently localized terms
\[
 \cut_{0} u^i\ \Rz_\beta [|\nabla^s u|^{p_s-2}\ \partial_\beta^s u^i] = I+II+III,
\]
Recall the definition \eqref{eq:def:CRWcommutator} of the commutator $\mathcal{C}(\cdot,\Rz)[\cdot]$, then 
\begin{align*}
%
 I :=\quad&\cut_{0} \Rz_\beta [\cut_{L}|\nabla^s u|^{p_s-2}\ u^i \partial_\beta^s u^i],\\
 II :=\quad&\cut_{0}\ \mathcal{C}(\cut_L (u^i-c^i),\ \Rz_\beta)[\cut_{L}|\nabla^s u|^{p_s-2}\ \partial_\beta^s u^i],\\
 III := \quad& \sum_{l=1}^\infty \cut_{0} u^i\ \Rz_\beta [\cutA_{L+l}|\nabla^s u|^{p_s-2}\ \partial_\beta^s u^i].
\end{align*}
Firstly, for $II$, by the usual Rochberg-Weiss commutator theorem,
\begin{align*}
 \vrac{II}_{p_s'} \aleq [\cut_L (u^i-c^i)]_{BMO}\ \vrac{\cut_L|\nabla^s u|^{p_s-1}}_{p_s'} = [\cut_L (u^i-c^i)]_{BMO}\ \vrac{\cut_L \nabla^s u}_{p_s}^{p_s-1}.
\end{align*}
Since $c^i$ was chosen arbitrarily, we can employ Lemma~\ref{la:BMOest} and then Lemma~\ref{la:RieszlapsII}, and obtain some $\sigma > 0$, some $S \geq 2$, for which
\begin{align*}
 \vrac{II}_{p_s'} \aleq \vrac{\cut_{L+S} \laps{s} u}_{p_s}^{p_s} + \vrac{\laps{s} u}_{p_s}\ \sum_{k=1}^\infty 2^{-(S+k)\sigma} \vrac{\cut_{L+S+k} \laps{s} u}_{p_s}^{p_s-1}.
\end{align*}
As for $I$,
\[
 \vrac{I}_{p_s'} \aleq \vrac{\cut_{L}\ \nabla^s u}_{p_s}^{p_s-2}\ \vrac{\cut_{L}\ u \cdot \partial_\beta^s u}_{p_s}.
\]
with the bi-commutator
\[
 H_T(a,b) := T(ab) - aTb - bTa,
\]
we then have
\[
 \chi_L u \cdot \partial_\beta^s u = -\frac{1}{2} \chi_L  H_{\partial_\beta^s}(u,u) - \chi_{L}\frac{1}{2} \partial_\beta^s |u|^2
\]
Since $\chi_{2L}|u| \equiv 1$ we can use the arguments on bi-commutators, in e.g. \cite{BPSknot12,DSpalphSphere}, to conclude that for some $\sigma > 0$,
\[
 \vrac{\cut_{L}\ u \cdot \partial_\beta^s u}_{p_s} \aleq \vrac{\cut_{L+S} \laps{s} u}_{p_s}^2 + 2^{-L\sigma}\vrac{\cut_{L+S} \laps{s} u}_{p_s} + \sum_{k=1}^\infty 2^{-(S+k)\sigma} \vrac{\cut_{L+S+k} \laps{s} u}_{p_s}^2,
\]
and consequently for possibly a smaller $\sigma$, using Lemma~\ref{la:RieszlapsII}, we have
\[
 \vrac{I}_{p_s'} \aleq \vrac{\cut_{L+S} \laps{s} u}_{p_s}^{p_s} + 2^{-L\sigma}\vrac{\cut_{L+S} \laps{s} u}_{p_s}+ \vrac{\laps{s} u}_{p_s}\ \sum_{k=1}^\infty 2^{-(S+k)\sigma} \vrac{\cut_{L+S+k} \laps{s} u}_{p_s}^{p_s-1}.
\]
It remains to analyze $III$, which we do by Lemma~\ref{la:local}
\begin{align*}
 \vrac{III}_{p_s'} 
 &\aleq \sum_{l=1}^\infty 2^{-\frac{n}{p_s'}(L+l)} \ \vrac{\cutA_{L+l}|\nabla^s u|^{p_s-1}}_{p_s'}\\
 &\aleq 2^{-\frac{n}{p_s'}L} \vrac{\cut_{L+S} \laps{s} u}_{p_s}^{p_s-1} + \sum_{l=1}^\infty 2^{-\sigma (L+S+l)} \ \vrac{\cutA_{L+l}\laps{s} u}^{p_s-1}_{p_s'}.\\
 \end{align*}
Taking again $S$ a multiple of $L$ and adjusting $\sigma$, the claim is established.
\end{proof}

For the $SO(N)$-case, we use essentially the same argument: If $\theta_{ij} = \theta_{ji}$ and $Q^{ki} Q^{kj} = \delta^{ij}$,
\begin{align*}
&\theta_{ij} Q^{ki} \Rz_{\alpha} [|\nabla^s Q|^{p_s-2} \partial_\alpha^s Q^{kj}]\\
=&\theta_{ij} \Rz_{\alpha} [|\nabla^s Q|^{p_s-2} Q^{ki} \partial_\alpha^s Q^{kj}]\\
&+\theta_{ij}\ \mathcal{C}(Q^{ki}, \Rz_{\alpha}) [|\nabla^s Q|^{p_s-2} \partial_\alpha^s Q^{kj}]\\
\overset{\theta \in sym}{=}&\frac{1}{2} \theta_{ij} \Rz_{\alpha} [|\nabla^s Q|^{p_s-2} \brac{Q^{ki} \partial_\alpha^s Q^{kj}+(\partial_\alpha^s Q^{ki})\ Q^{kj} }]\\
&+\theta_{ij}\ \mathcal{C}(Q^{ki}, \Rz_{\alpha}) [|\nabla^s Q|^{p_s-2} \partial_\alpha^s Q^{kj}]\\
\overset{Q^TQ\equiv I}{=}&\frac{1}{2} \theta_{ij} \Rz_{\alpha} [|\nabla^s Q|^{p_s-2} \brac{Q^{ki} \partial_\alpha^s Q^{kj}+(\partial_\alpha^s Q^{ki})\ Q^{kj} - \partial_\alpha^s(Q^{ki}  Q^{kj})}]\\
&+\theta_{ij}\ \mathcal{C}(Q^{ki}, \Rz_{\alpha}) [|\nabla^s Q|^{p_s-2} \partial_\alpha^s Q^{kj}]\\
=&\frac{1}{2} \theta_{ij} \Rz_{\alpha} [|\nabla^s Q|^{p_s-2}\ H_{\partial^s_\alpha} (Q^{ki},Q^{kj})]\\
&+\theta_{ij}\ \mathcal{C}(Q^{ki}, \Rz_{\alpha}) [|\nabla^s Q|^{p_s-2} \partial_\alpha^s Q^{kj}]\\
\end{align*}
Thus, we obtain the following

\begin{lemma}\label{la:orthogonalpartSO}
For any $K \in \Z$, $L \in \N$, $L \geq L_0$, and any $u$ the following holds:
If $\chi_{2L} u^T u \equiv 1$, then for some uniform $\sigma > 0$,
\begin{align*}
&\|\cut_{K} \theta_{ij}u^{ki}\ \Rz_\beta [|\nabla^s u|^{p_s-2}\ \partial_\beta^s u^{kj}] \|_{p_s'}\\
\aleq& \brac{\vrac{\cut_{K+L} \laps{s} u}_{p_s} + 2^{-L\sigma}}\ \vrac{\cut_{K+L} \laps{s} u}_{p_s}^{p_s-1}\\
&+ (1+\vrac{\laps{s} u}_{p_s})\ \sum_{k=1}^\infty 2^{-(L+k)\sigma} \vrac{\cut_{K+L+k} \laps{s} u}_{p_s}^{p_s-1}.
\end{align*}
\end{lemma}


\section{argument: tangential part}
\begin{lemma}\label{la:tangentialpart}
Let $u$ be as in Theorem~\ref{th:mainS}. For any $\omega_{ij} = -\omega_{ji} \in \{-1,0,1\}$, any smooth $\supp \varphi \subset \supp \cut_L$, $L \in \Z$
\begin{align*}
&\omega_{ij} \int \laps{s} \varphi\ u^j\ \ \Rz_\beta [|\nabla^s u|^{p_s -2} \Rz_\beta [\laps{s} u^i]] \\
\aleq & \vrac{\cut_{L+K}\laps{s} u}_{p_s}^{p_s} + \vrac{\laps{s}u}_{p_s,\R^n} \sum_{k=1} 2^{-\sigma(k+L+K)} \vrac{\cut_{L+K+k} \laps{s}u}_{p_s}^{p_s-1}\\
\end{align*}
\end{lemma}
\begin{proof}
Again we show the claim for $\cut_0$ instead of $\cut_L$.
The Euler-Lagrange equations imply that for any $\varphi \in C_0^\infty(\Omega)$, we have for $\omega_{ij} = -\omega_{ji} \in \{-1,0,1\}$
\begin{equation}\label{eq:EKeqn}
 \omega_{ij} \int_{\R^n} \Rz_\beta [|\nabla^s u|^{p_s -2} \Rz_\beta [\laps{s} u^i]]\ \laps{s} (u^j \varphi) = 0.
\end{equation}
This is true, since $(u^i \omega_{ij} \varphi)_{j} \in \R^N$ is perpendicular to $T_u \S^{N-1}$ and using partial integration of $\Rz_\beta$.

On the other hand, by the antisymmetrie of $\omega$,
\begin{equation}\label{eq:OtherZero}
 \omega_{ij} \int_{\R^n} |\nabla^s u|^{p_s -2} \Rz_\beta [\laps{s} u^i]\ \Rz_\beta [\laps{s} u^j] \varphi = 0.
\end{equation}
Consequently, recalling the definition of the commutator $\mathcal{C}(\cdot,T)[\cdot)$ in \eqref{eq:def:CRWcommutator} and the bi-commutator $H_s$ defined in \eqref{eq:defHs},
\begin{align*}
 &\omega_{ij} \int \laps{s} \varphi\ u^j\ \ \Rz_\beta [|\nabla^s u|^{p_s -2} \Rz_\beta [\laps{s} u^i]]\\
 =&-\omega_{ij}  \int H_s(\varphi,u^j)\ \ \Rz_\beta [|\nabla^s u|^{p_s -2} \Rz_\beta [\laps{s} u^i]]\\
 &- \omega_{ij} \int \cut_{2K} \mathcal{C}(\varphi,\Rz_\beta) [\cut_{K}\laps{s} u^j]\ \ |\nabla^s u|^{p_s -2} \Rz_\beta [\laps{s} u^i]\\
 &- \sum_{l=1}^\infty \omega_{ij} \int \cut_{0} \varphi \Rz_\beta [\cutA_{K+l}\laps{s} u^j]\ \ |\nabla^s u|^{p_s -2} \Rz_\beta [\laps{s} u^i]\\
 & + \sum_{k=1}^\infty \omega_{ij} \int \cutA_{2K+k} \Rz_\beta [\varphi\ \laps{s} u^j]\ \ |\nabla^s u|^{p_s -2} \Rz_\beta [\laps{s} u^i]\\
 &=: I + II + III + IV.
\end{align*}
For the first term, by the usual arguments on bi-commutators, e.g., as in \cite[Lemma 3.2.]{BPSknot12}, \cite[Lemma 2.8]{DSpalphSphere}, using also again Lemma~\ref{la:RieszlapsII}, we obtain readily,	
\[
 |I| \aleq \vrac{\cut_K\laps{s} u}_{p_s}^{p_s} + \vrac{\laps{s}u}_{p_s,\R^n} \sum_{k=1} 2^{-\sigma(k+K)} \vrac{\cut_{K+k} \laps{s}u}_{p_s}^{p_s-1}
\]
For the second term, we use again the Rochberg-Weiss commutator theorem \cite{CRW76}, and then again Lemma~\ref{la:RieszlapsII},
\begin{align*}
 |II| &\aleq   \vrac{\cut_{2K}\nabla^s u}_{p_s}^{p_s-1}\ [\varphi]_{BMO}\ \vrac{\cut_{K}\laps{s} u}_{p_s}\\
 &\aleq \vrac{\cut_K\laps{s} u}_{p_s}^{p_s} + \vrac{\laps{s}u}_{p_s,\R^n} \sum_{k=1} 2^{-\sigma(k+K)} \vrac{\cut_{K+k} \laps{s}u}_{p_s}^{p_s-1}.
\end{align*}
As for $III$, by Poincar\'{e} inequality and the localization Lemma~\ref{la:local}, and yet again Lemma~\ref{la:RieszlapsII},
\begin{align*}
 |III| &\aleq  \sum_{l=1}^\infty \int \cut_{0} \varphi \Rz_\beta [\cutA_{K+l}\laps{s} u^j]\ \ |\nabla^s u|^{p_s -2} \Rz_\beta [\laps{s} u^i]\\
 &\aleq  \sum_{l=1}^\infty \vrac{\cut_0 \Rz_\beta [\cutA_{K+l}\laps{s} u^j]}_{\infty}\ \vrac{\varphi}_{p_s}\ \vrac{\cut_0 \nabla^s u}_{p_s}^{p_s-1}\\
 &\aleq  \sum_{l=1}^\infty 2^{-(K+l)\frac{n}{p_s}}\ \vrac{\cutA_{K+l}\laps{s} u^j}_{p_s}\ \vrac{\cut_0 \nabla^s u}_{p_s}^{p_s-1}\\
 &\aleq \vrac{\cut_{2K}\laps{s} u}_{p_s}^{p_s} + \vrac{\laps{s}u}_{p_s,\R^n} \sum_{k=1} 2^{-\sigma(k+2K)} \vrac{\cut_{2K+k} \laps{s}u}_{p_s}^{p_s-1}.
\end{align*}
Finally, we treat $IV$, now using additionally Sobolev-Poincar\'{e}-inequality
\begin{align*}
 |IV| \aleq & \sum_{k=1}^\infty  2^{-(K+k+k)\frac{n}{p_s'}}\ R^{-\frac{n}{p_s'}}\ \vrac{\varphi\ \laps{s} u^j}_1\ \ \vrac{\cut_{2K+k}\nabla^s u}_{p_s}^{p_s -1} \\
\aleq & \sum_{k=1}^\infty  2^{-(K+k+k)\frac{n}{p_s'}}\ R^{-\frac{n}{p_s'}}\ \vrac{\varphi}_{p_s'}\ \vrac{\cut_0 \laps{s} u^j}_{p_s}\ \ \vrac{\cut_{2K+k}\nabla^s u}_{p_s}^{p_s -1} \\
\aleq & \sum_{k=1}^\infty  2^{-(K+k+k)\frac{n}{p_s'}}\   \vrac{\laps{s} \varphi}_{p_s'}\ \vrac{\cut_0 \laps{s} u^j}_{p_s}\ \ \vrac{\cut_{2K+k}\nabla^s u}_{p_s}^{p_s -1}.
\end{align*}

\end{proof}

Again in the same fashion we obtain an estimate for the case of $SO(N)$-target. 
The Euler-Lagrange equation tells us that a critical point $u: \Omega \to SO(N)$ satisfies
\begin{equation}\label{eq:ELeq}
  \int |\nabla^s u|^{p_s-2} \partial_\alpha^s u_{kj}\ \nabla^s \psi_{kj} = 0,\\
\end{equation}
for any $\psi$ with support in $\Omega$ and a.e. $\psi \in T_u SO(N)$. For antisymmetric $\omega \in so(N)$, $\varphi \in C_0^\infty(\Omega)$, we thus set
\[
 \psi_{kj} := Q^{ki} \omega_{ij} \varphi.
\]
Then $\psi \in T_{Q}SO(N)$. Indeed, $\nu: (-1,1) \to SO(N)$ defined as $\nu(t) := ue^{t\omega \varphi}$ has $\nu'(0) = \psi$. That is, \eqref{eq:ELeq} implies for any $\varphi \in C_0^\infty(D)$, $\omega \in so(N)$,
\begin{equation}\label{eq:ourELeq}
  \omega_{ij}  \int |\nabla^s u|^{p_s-2} \partial_\alpha^s u^{kj}\ \nabla^s (u^{ki} \varphi) = 0.
\end{equation}
Moreover, by the antisymmetry of $\omega$,
\begin{equation}\label{eq:antisymmetrykills}
 \partial_\alpha^s Q^{kj}\ \omega_{ij} \partial_\alpha^s Q^T_{ik} = 0.
\end{equation}

Consequently, for $\omega \in so(N)$, using in the last step \eqref{eq:ourELeq} and \eqref{eq:antisymmetrykills},
\begin{align*}
& \int \omega_{ij} Q^T_{ik} \Rz_{\alpha} [|\nabla^s Q|^{p_s-2} \partial_\alpha^s Q^{kj}]\ \laps{s} \varphi\\
=&-\int |\nabla^s Q|^{p_s-2} \partial_\alpha^s Q^{kj}\ \Rz_\alpha [\omega_{ij} Q^T_{ik} \laps{s} \varphi]\\
=&-\omega_{ij} \int |\nabla^s Q|^{p_s-2} \partial_\alpha^s Q^{kj}\ Q^T_{ik} \partial_\alpha^s \varphi\\
&-\omega_{ij}\int |\nabla^s Q|^{p_s-2} \partial_\alpha^s Q^{kj}\ \mathcal{C}( Q^T_{ik},\Rz_\alpha) [\laps{s} \varphi]\\
=&\omega_{ij} \int |\nabla^s Q|^{p_s-2} \partial_\alpha^s Q^{kj}\ H_{\partial_\alpha^s}(Q^T_{ik}, \varphi)\\
&-\omega_{ij}\int |\nabla^s Q|^{p_s-2} \partial_\alpha^s Q^{kj}\ \mathcal{C}( Q^T_{ik},\Rz_\alpha) [\laps{s} \varphi]\\
\end{align*}
From this argument, we obtain readily
\begin{lemma}\label{la:tangentialpartSO}
Let $u$ be as in Theorem~\ref{th:mainSO}. For any $\omega_{ij} = -\omega_{ji} \in \{-1,0,1\}$, any smooth $\supp \varphi \subset \supp \cut_L$, $L \in \Z$
\begin{align*}
&\omega_{ij} \int \laps{s} \varphi\ u^{ik}\ \ \Rz_\beta [|\nabla^s u|^{p_s -2} \Rz_\beta [\laps{s} u^{kj}]] \\
\aleq & \vrac{\cut_{L+K}\laps{s} u}_{p_s}^{p_s} + \vrac{\laps{s}u}_{p_s,\R^n} \sum_{k=1} 2^{-\sigma(k+L+K)} \vrac{\cut_{L+K+k} \laps{s}u}_{p_s}^{p_s-1}\\
\end{align*}
\end{lemma}

%
%
%
%

\begin{appendix}
\section{Some basic Estimates}
Recall \eqref{eq:cutdef}. In this section we recall some estimates involving $\laps{s}$-operators, in particular in relation to (pseudo-)localization.

The most basic estimate is the following localization. It follows quite naturally from the potential definition of the involved operators. For details we refer, e.g., to the appendix of \cite{BPSknot12}.
\begin{lemma}\label{la:local}
Let $s \in (-n,n)$, and if $s > 0$, and $T^s$ defined as follows. \begin{itemize} \item if $s > 0$, $T^s = \nabla^s$ or $T^s = \laps{s}$ \item if $s = 0$, $T^0 = \Rz_\alpha$, for any $\alpha \in \{1,\ldots,n\}$,
\item and if $s < 0$, $T^s = \lapms{s}$. \end{itemize}
Then, $l \geq k+1$, for any $f$,
\[
 \vrac{\cutA_l T^s[\cut_k f]}_\infty \aleq (2^k)^{-n-s} \vrac{\cut_k f}_1
\]
and
\[
 \vrac{\cut_k T^s[\cutA_l f]}_\infty \aleq (2^l)^{-n-s} \vrac{\cutA_l f}_1
\]
\end{lemma}

The following estimate estimates an $L^p$-norm by an elliptic PDE. For the proof we refer to, e.g., \cite[Proposition A.3.]{BPSknot12}.
\begin{lemma}[Estimate by PDE]\label{la:EstimatebyPDE}
Let $f \in L^{p}(\R^n)$, then
\[
 \vrac{\cut_L f}_{p,B_R} \aleq \sup_{\varphi} \int_{\R^n} f \laps{s} \varphi + 2^{-\sigma K} \vrac{\cut_{L+K}f}_{p} + \sum_{k =1}^\infty 2^{-\sigma(K+l)}\ \vrac{\cut_{L+K+k} f}_{p},
\]
where the supremum is taken over all $\varphi \in C^\infty(\R^n)$ with $\supp \varphi \subset \supp \cut_{L}$ and $\vrac{\laps{s} \varphi}_{p'} \leq 1$.
\end{lemma}

Let $\Rz = (\Rz_1,\ldots,\Rz_n)$ be the vector of all Riesz transforms. It is well-known that the norm $\vrac{f}_{p,\R^n}$ and $\vrac{\Rz f}_{p,\R^n}$ are equivalent for $p \in (1,\infty)$. This is certainly not true anymore if the norm is taken on strict subsets of $\R^n$. However Lemma~\ref{la:local} provides a comparison, which tells us that
\[
 \vrac{f}_{p,B_r}\quad  \mbox{more or less} \quad  \vrac{\Rz f}_{p,B_r},
\]
with error estimates. More precisely, we have Lemma~\ref{la:lapstoRiesz}, Lemma~\ref{la:RieszlapsII}, and Lemma~\ref{la:lapstoRieszII}. The main idea is always that $f = c\ \Rz_\beta \Rz_\beta f$; then one inserts the cutoff-functions $\cut_{i}$ and uses Lemma~\ref{la:lapstoRiesz}.

\begin{lemma}\label{la:lapstoRiesz}
Let $f \in L^p(\R^n)$ and $K \geq L \in \Z$, $M \geq 2$
\begin{align*}
 \|(\cut_{K} - \cut_L)f\|_{p} &\aleq  \|(\cut_{K+M}-\cut_{L-M})\ \Rz f \|_{p} \\
 & + \sum_{k=1}^\infty 2^{-(M+k)\frac{n}{p}}\ \|\cutA_{K+M+k}\ \Rz [f]\|_{p} \\
 & + 2^{(K-L)\frac{n}{p}-M\frac{n}{p'}}\  \|\cut_{L-M}\ \Rz [f]]\|_{p} \\ \\
\end{align*}
\end{lemma}
\begin{proof}
We have
\begin{align*}
 \|(\cut_{K} - \cut_L)\ f\|_{p}
 \aeq &\|(\cut_{K} - \cut_L)\ \Rz_\beta \Rz_\beta [f]\|_{p}\\
 \leq  &\|(\cut_{K} - \cut_L)\ \ \Rz_\beta [(\cut_{K+M}-\cut_{L-M})\ \Rz_\beta [f]]\|_{p} \\
& + \sum_{k=1}^\infty \|(\cut_{K} - \cut_L)\ \ \Rz_\beta [(\cutA_{K+M+k}\ \Rz_\beta [f]]\|_{p} \\
 & + \|(\cut_{K} - \cut_L)\ \ \Rz_\beta [(\cut_{L-M}\ \Rz_\beta [f]]\|_{p}.
\end{align*}
Now the claim follows from Lemma~\ref{la:local}
\end{proof}

\begin{lemma}\label{la:lapstoRieszII}
Let $f \in L^p(\R^n)$, 
\begin{align*}
 \|\cut_L f\|_{p} &\aleq  \|\cut_{L+S} \Rz [f]\|_{p} + \sum_{l=1}^\infty l\ 2^{-\frac{n}{p}(S+l)}\|\cutA_{S+L+l} f\|_{p}.
\end{align*}
\end{lemma}
\begin{proof}
W.l.o.g., $L=0$. Again we have
\begin{align*}
 \|\cut_0 f\|_{p} &
  \leq \|\cut_0 \Rz_\beta [\cut_S \Rz_\beta [f]]\|_{p} + \sum_{l=1}^\infty \|\cut_0 \Rz_\beta [\cutA_{S+l} \Rz_\beta [f]]\|_{p}\\
 &\aleq \|\cut_S \Rz [f]\|_{p} + \sum_{l=1}^\infty 2^{-\frac{n}{p}(S+l)}\|\cutA_{S+l} \Rz [f]\|_{p}.
\end{align*}
We still need to remove the $\Rz[\cdot]$ in the second term.
\begin{align*}
\|\cutA_{S+l} \Rz [f]\|_{p} &\aleq
\|\cutA_{S+l} \Rz [\cut_{S+l+1} f]\|_{p} + \sum_{k=2}^\infty \|\cut_{S+l} \Rz [\cutA_{S+l+k} f]\|_{p}\\
&\aleq \|\cut_{S+l+1} f\|_{p} + \sum_{k=2}^\infty 2^{-k\frac{n}{p}}\|\cutA_{S+l+k}f\|_{p}.
\end{align*}
Consequently,
\begin{align*}
 \|\cut_0 f\|_{p} &\aleq \|\cut_S \Rz [f]\|_{p} + \sum_{l=1}^\infty 2^{-\frac{n}{p}(S+l)}  \|\cut_{S+l+1} f\|_{p} + \sum_{l=1}^\infty 2^{-\frac{n}{p}(S+l)} \sum_{k=2}^\infty 2^{-k\frac{n}{p}}\|\cutA_{S+l+k}f\|_{p}.
\end{align*}
Finally, we observe
\begin{align*}
 &\sum_{l=1}^\infty 2^{-\frac{n}{p}(S+l)} \sum_{k=2}^\infty 2^{-k\frac{n}{p}}\|\cutA_{S+l+k}f\|_{p}\\
 =&\sum_{l=1}^\infty 2^{-\frac{n}{p}(S+l)} \sum_{k=l+2}^\infty 2^{-(k-l)\frac{n}{p}}\|\cutA_{S+k}f\|_{p}\\
 =&\sum_{k=3}^\infty 2^{-(k+S)\frac{n}{p}}\ \|\cutA_{S+k}f\|_{p} \sum_{l=1}^{k-2} 1\\
 \leq &\sum_{k=3}^\infty (k-2)\ 2^{-(k+S)\frac{n}{p}}\ \|\cutA_{S+k}f\|_{p}.
\end{align*}
\end{proof}

By the same arguments as above, we also have
\begin{lemma}\label{la:RieszlapsII}
Let $f \in L^p(\R^n)$, $K \in \mathbb{Z},\ L \in \mathbb{N}$
\[
 \|\cut_{L}\Rz f\|_{p} \leq C\ \|\cut_{L+K} f \|_{p} + C\ \sum_{k=1}^\infty 2^{-\frac{n}{p}(K+k)} \|\cutA_{L+K+k} f \|_{p}.
\]
\end{lemma}

\section{BMO-Estimates}
In the estimates we sometimes have to estimate the BMO of $\cut (u-(u))$. Let us recall the definition of the pseudo-norm for $BMO$:
\[
 [f]_{BMO} := \sup_{B_r \subset \R^n} |B_r|^{-1} \int_{B_r} \left |f(x) -  |B_r|^{-1} \int_{B_r} f(y) dy \right |\ dx.
\]

We will use the notation \eqref{eq:cutdef}. For simplicity of presentation, we assume throughout this section that $\cut_0 = \chi_{B_1(0)}$.

Firstly, we have the following
\begin{proposition}\label{pr:BMOestmain}
 \begin{align*}
 \sup_{B_r \subset B_2(0), r \leq 1} r^{-2n} \int_{B_r} \int_{B_r} |u(x)-u(y)| +
 \sup_{B_r \subset B_2(0), r \leq 1} r^{-n} \int_{B_r} \int_{B_1} |u(x)-u(y)| \\
 \aleq \vrac{\cut_{B_{2^K}(0)}|\laps{s} u|}_{p_s} + \sum_{l=1}^\infty 2^{-(K+l)} \vrac{\cut_{B_{2^{K+l}}(0)}\laps{s} u }_{p_s}
 \end{align*}
\end{proposition}
\begin{proof}
By definition of the Riesz-potential \[
                                      \lapms{s} f(x) = c\ \int |z-x|^{s-n} f(z),
                                     \]
and since $\lapms{s} \laps{s} = Id$, we have for any $K \geq 3$
\begin{align*}
 |u(x)-u(y)| &\aleq \int_{\R^n} ||x-z|^{s-n} - |y-z|^{s-n}|\ |\laps{s} u(z)|\ dz\\
 &= \int_{\R^n} ||x-z|^{s-n} - |y-z|^{s-n}|\ \cut_{K}(z)|\laps{s} u(z)|\ dz\\
 &\quad + \sum_{l=1}^\infty \int_{\R^n} ||x-z|^{s-n} - |y-z|^{s-n}|\ \cutA_{K+l}(z)|\laps{s} u(z)|\ dz\\
 &=: I(x,y) + II(x,y).
\end{align*}
As for $II(x,y)$: for $x,y \in B_2(0)$, and $z \in \supp \cutA_{K+l}$ we have
\[
 ||x-z|^{s-n} - |y-z|^{s-n}| \aleq 2^{K(s-n-1))}|x-y| \leq 2^{K(s-n-1)} |x-y|,
\]
that is, since $p_s' = \frac{n}{n-s}$,
\[
 II(x,y) \leq |x-y|\ 2^{-(K+l)} \vrac{\cut_{K+l}\laps{s} u }_{p_s}.
\]
Consequently,
\begin{align*}
 &r^{-2n} \int_{B_r} \int_{B_r} |II(x,y)| + r^{-n} \int_{B_r} \int_{B_1} |II(x,y)|\\
 &\aleq \sum_{l=1}^\infty 2^{-(K+l)} \vrac{\cut_{K+l}\laps{s} u }_{p_s}
\end{align*}
It remains to estimate $I$. As in \cite{Sfracenergy}, see also the presentation in \cite[Lemma A.4.]{BPSknot12}, we have for almost any $x,y,z$,
\begin{align*}
 ||x-z|^{s-n} - |y-z|^{s-n}| \leq &\chi_{|x-y| \aleq |x-z|,|y-z|} |x-y|^1 \min\{|y-z|^{s-n},|x-z|^{s-n}\}\\
 &\quad + \chi_{|y-z| < |x-y|} |y-z|^{s-n}\\
 &\quad + \chi_{|x-z| < |x-y|} |x-z|^{s-n}.
\end{align*}
Thus,
\begin{align*}
  &r^{-2n} \int_{B_r} \int_{B_r} |I(x,y)|\\
  \aleq & r^{-2n}\ \int_{B_r} \int_{B_r} |x-y| \int_{z:|x-z|>|x-y|}\ |x-z|^{s-1-n}  \cut_{K}(z)|\laps{s} u(z)|\ dz\ dy\ dx\\
  & + 2r^{-2n}\ \int_{B_r} \int_{B_{3r}} \int_{B_r \cap |y-z| < |x-y| < 2r}  \ |y-z|^{s-n}  \cut_{K}(z)|\laps{s} u(z)|\ dy\ dz\ dx\\
  \aleq & r^{-2n}\ \int_{B_r} \int_{B_r} |x-y| \brac{\int_{z:|x-z|>|x-y|}\ |x-z|^{(s-1-n)\frac{n}{n-s}}}^{\frac{1}{p_s'}} \vrac{\cut_{K}|\laps{s} u|}_{p_s}\ dy\ dx\\
  & + r^{s-n}\ \int_{B_{3r}}  \cut_{K}(z)|\laps{s} u(z)|\  dz\\
   \aleq & r^{-2n}\ \int_{B_r} \int_{B_r} |x-y|\ |x-y|^{-1}\ \vrac{\cut_{K}|\laps{s} u|}_{p_s}\ dy\ dx\\
  & + r^{s-n}\ \int_{B_{3r}}  \cut_{K}(z)|\laps{s} u(z)|\  dz\\
   \aleq & r^{-2n}\ \int_{B_r} \int_{B_r} |x-y|\ |x-y|^{-1}\ \vrac{\cut_{K}|\laps{s} u|}_{p_s}\ dy\ dx\\
  & + r^{s-n}\ \int_{B_{3r}}  \cut_{K}(z)|\laps{s} u(z)|\  dz\\
  \aleq & \vrac{\cut_{K}|\laps{s} u|}_{p_s}.
\end{align*}

\end{proof}

\begin{lemma}[BMO Estimate]\label{la:BMOest}
For any $u$ and $L$ there exists a constant $c \in \R$ such that
\[
[\cut_L (u-c)]_{BMO} \leq C\ \vrac{\cut_{L+K}|\laps{s} u|}_{p_s} + C\ \sum_{l=1}^\infty 2^{-(K+l)} \vrac{\cut_{L+K+l}\laps{s} u }_{p_s}
\]
\end{lemma}
\begin{proof}
Again we may assume that $L=0$. Let $c:=(u)_0$ is then the mean value of $u$ over $B_1(0)$.
\begin{align*}
 &[\cut_0 (u-(u)_0)]_{BMO}\\
  \aleq& \sup_{B_r \subset B_2(0), r \leq 1} r^{-2n} \int_{B_r} \int_{B_r} |\cut_0(x)(u(x)-(u)_0)-\cut_0(y)(u(y)-(u)_0)|\\
  &+ \sup_{B_r \subset B_2(0), r \leq 1} r^{-n} \int_{B_r} \int_{B_1(0)} |u(x)-u(y)|.
\end{align*}
Since
\begin{align*}
 &|\cut_0(x)(u(x)-(u)_0)-\cut_0(y)(u(y)-(u)_0)|\\
 \leq\ & |u(x)-(u)_0|+|u(x)-u(y)| + |u(y)-(u)_0|.
\end{align*}
this becomes
\begin{align*}
 &[\cut_0 (u-(u)_0]_{BMO}\\
   \aleq& \sup_{B_r \subset B_2(0), r \leq 1} r^{-2n} \int_{B_r} \int_{B_r} |u(x)-u(y)|\\
  &+ \sup_{B_r \subset B_2(0), r \leq 1} r^{-n} \int_{B_r} \int_{B_1(0)} |u(x)-u(y)|.
\end{align*}
Now the claim follows from Proposition~\ref{pr:BMOestmain}.
\end{proof}

\section{Decompositon for SO(N)}
In \eqref{eq:nablaudecomp} we decompose vectors into orthogonal and tangential part along $u \in S^{N-1}$. The analogue for $Q \in SO(N)$ is the following
\begin{proposition}\label{pr:decompSO}
There exists a uniform constant $C > 0$ such that the following holds for any matrix $A \in \R^{N \times N}$, for any $Q \in SO(N)$,
\[
 |A| \leq \max_{\omega} |\omega_{ij} Q^{ik}A_{kj}| |+ \max_{\sigma} |\theta_{ij} Q^{ik} A_{kj}| 
\]
where the maximum is taken on all finitely many $\omega \in \{-1,0,1\}^{N \times N} \in so(N)$ and $\sigma \in \{0,1\}^{N \times N} \in sym(N)$, respectively.
\end{proposition}
\end{appendix}

\bibliographystyle{plain}%
\bibliography{bib}%

\end{document}